\documentclass[12pt,a4paper]{amsart}

\usepackage[margin=2cm]{geometry}

\usepackage{enumerate}

\usepackage{amsthm}
\usepackage{amsmath}
\usepackage{amsfonts}
\usepackage{amssymb}
\usepackage{amsrefs}
\usepackage[ocgcolorlinks,hyperfootnotes=false,colorlinks=true,citecolor=blue,linkcolor=blue,urlcolor=blue]{hyperref}

\newcommand{\abs}[1]{\left\lvert {#1} \right\rvert}

\newcommand{\C}{\mathbb{C}}
\newcommand{\eC}{\widehat{\mathbb{C}}}

\newcommand{\D}{\mathbb{D}}

\newcommand{\E}{\mathbb{E}}
\newcommand{\eE}{\widehat{E}}
\newcommand{\Ec}{\mathbb{C}\setminus E}

\renewcommand{\O}{\mathcal{O}}

\newcommand{\aut}{\operatorname{Aut}}

\newtheorem{thm}{Theorem}[section]

\newtheorem{cor}[thm]{Corollary}
\newtheorem{lemma}[thm]{Lemma}

\theoremstyle{definition}
\newtheorem{defn}[thm]{Definition}

\title{Conformal Rigidity of Polar Set Complements}
\date{February 18, 2022}


\makeatletter
\@namedef{subjclassname@2020}{%
  \textup{2020} Mathematics Subject Classification}
\makeatother

\author{Ratna Pal}
\address{Indian Institute of Science Education and Research, Berhampur 760010, India}
\email{ratnap@iiserbpr.ac.in}

\author{Koushik Ramachandran}
\address{Tata Institute of Fundamental Research, Centre for Applicable Mathematics, Bengaluru 560065, India}
\email{koushik@tifrbng.res.in}

\author{Sivaguru Ravisankar}
\address{Tata Institute of Fundamental Research, Centre for Applicable Mathematics, Bengaluru 560065, India}
\email{sivaguru@tifrbng.res.in}

\begin{document}

\begin{abstract}
Let $E$ be a closed polar subset of $\C$.
In this short note, we use elementary potential theoretic tools to show that any
conformal map on $\C\setminus{E}$ is necessarily a  M\"{o}bius map.
As a consequence we obtain that the group of conformal automorphisms of the complement of a closed polar set $E$ is a discrete subgroup of the  M\"{o}bius group, provided $\lvert E \rvert \ge 2$.
\end{abstract}

\maketitle

\section{Introduction}
Let $\Omega_1$ and $\Omega_2$ be domains in the complex plane. Let $\mathcal{A}\left(\Omega_1, \Omega_2\right)$ denote the space of conformal maps from $\Omega_1$ to $\Omega_2$. Recall that a map $f:\Omega_1\rightarrow \Omega_2$ is said to be conformal if it is bijective and holomorphic. If $\mathcal{A}\left(\Omega_1, \Omega_2\right)\neq\emptyset,$ we say that $\Omega_1$ and $\Omega_2$ are conformally equivalent domains. Understanding when two domains are conformally equivalent is a classical theme of study in complex analysis.The crowning achievements in this theory are the celebrated Riemann mapping theorem and its extension, the uniformization theorem due to Koebe.
The Riemann mapping theorem tells us that any proper simply connected domain in $\mathbb{C}$ is conformally equivalent to the unit disc.
On the other hand Koebe proved that any domain of finite connectivity is conformally equivalent to a circle domain. A domain is said to be a {\it circle domain} if every connected component of its boundary is either a single point or a circle.

A closely related problem is the study of $\aut(\Omega),$ the space 
of all automorphisms of $\Omega$ by which we mean the space of all conformal self-maps of $\Omega$. A standard fact that one sees in a first course in complex analysis is that $\aut(\D)$ consists of the Blaschke factors, $\aut(\C)$ consists of linear polynomials, and $\aut(\C\setminus\{0\})$ consists of multiples of $z$ and $1/z$. Koebe's theorem implies that a domain of connectivity two is conformally equivalent to an annulus and it is elementary to see that any automorphism of an annulus is either a rotation or an inversion.
For a finitely connected domain of connectivity at least three, Heins \cite{Heins}
showed that every automorphism extends as a M\"{o}bius map to the extended complex plane.
Furthermore, there are only finitely many automorphisms.

Short \cite{Short} proved that the automorphism group of a countably connected domain is either a Fuchsian group or a discrete elementary group of M\"{o}bius maps.
Short's result relies on a deep theorem of He--Schramm \cite{HeSch}  which
states that any countably connected domain is conformally equivalent to a circle domain.
This result of He--Schramm marks a significant advancement towards K\"{o}be's conjecture that any planar domain is conformally equivalent to a circle domain.
K\"{o}be's conjecture is still open for uncountably connected domains. 

In this article, we consider conformal maps on the  complement of a closed polar set in $\C$.
Polar sets are negligible sets in potential theory; they are the $-\infty$ level sets of subharmonic functions.
We note that there are uncountable polar sets and, since polar sets are totally disconnected (see \cite{Ra}*{Section 5.3}), their complements furnish examples of circle domains with uncountably many connected components.
We prove that any conformal map on the complement of a closed polar set in $\mathbb{C}$ is the restriction of a M\"{o}bius map.

Our result follows from a rigidity theorem of He-Schramm \cite{HeSch1}, which states that any conformal map from a circle domain, whose boundary has $\sigma$-finite linear measure, to any other circle domain is the restriction of a M\"{o}bius map.
It is well known that polar sets have zero linear measure (see \cite{Carleson}*{\S IV, Theorem 1}).
The proof of the aforementioned rigidity theorem uses advanced tools from geometric function theory.
In contrast, we provide here an elementary proof of a special case using only basic tools from complex analysis and potential theory.
We hope that this short note makes the rigidity result more accessible to a wider audience.

For related results and ideas we refer the interested reader to Ahlfors and Beurling \cite{AhlBeu}, Havinson \cite{Havinson}, Younsi \cite{Younsi}, and the references therein.

\subsection*{Acknowledgements}
We thank Anne-Katrin Gallagher, Dmitry Khavinson, Jiri Lebl, and Malik Younsi for helpful comments and feedback on the article. 

\section{Definitions and Notations}
We denote the extended complex plane by $\eC=\C\cup\{\infty\}$.
For a set $E\subset\C$, let $\eE=E\cup\{\infty\}$.
The space of holomorphic functions on a domain $\Omega$ will be denoted by $\O(\Omega)$.
For domains $\Omega_1$ and $\Omega_2$ in $\C$, $\mathcal{A}\left(\Omega_1, \Omega_2\right)$ denotes the space of conformal maps from $\Omega_1$ to $\Omega_2$.
We next recall some notions from potential theory.

\begin{defn}\label{energy}
Let $\mu$ be a finite Borel measure on $\C$ with compact support. The energy of the measure $\mu$ is defined by
$$I(\mu) = \int\int\log|z-w|d\mu(z)d\mu(w).$$
\end{defn}

\begin{defn}\label{defn:polar}
A subset $E$ of $\C$ is called \emph{polar} if $I(\mu) = - \infty$ for every finite Borel measure $\mu\neq 0$ for which the support of $\mu$ is a compact subset of $E$.
\end{defn}

It is easily checked that every finite or countable subset of $\C$ is polar. However, there are uncountable polar sets.
For a beautiful account of potential theory in the plane and its applications to complex analysis, we refer to the book by Ransford \cite{Ra}.

We use the following generalized notion of singularities of a holomorphic function.
\begin{defn}
Let $E$ be a closed polar subset of $\C$ and $f\in\O(\Ec)$.
We call points of $\eE$ as {\it singularities} of $f$.
The singularity $z_0\in\eE$ of $f$ is said to be
\begin{enumerate}[\qquad (i)]
	\item {\it removable} if $f$ is bounded on $U\cap (\C\setminus\E)$ for some neighbourhood $U$, in $\eC$, of $z_0$,
	\item {\it pole-like} if for every sequence $(z_n)$ in $\Ec$ converging to $z_0$, 
	\[f(z_n) \to \infty,\quad\text{and}\]
	\item {\it essential-like} if there are sequences $(z_n)$ and $(\widetilde{z}_n)$
	in $\Ec$ converging to $z_0$, and $\ell\in\C$ such that 
	\[f(z_n) \to \ell \quad\text{and}\quad f(\widetilde{z}_n) \to \infty.\]
	\end{enumerate}
\end{defn}
Note that every singularity of $f$ falls into one of the categories above.
The definition above does not require the singularity to be isolated.
This apparent relaxation for the notion of a removable singularity is justified by Lemma \ref{lemma:Removable}.
We drop the suffix ``-like'' when the singularity is isolated.

\section{Main Result}

We state our main result and make a few remarks here.
The proof is contained in Section~\ref{sec:proofs}.
\begin{thm}\label{thm:CptMainResult}
Let $E$ be a closed polar subset of $\C$ and $f$ be a conformal map on $\Ec$.
Then $f$ extends to a M\"{o}bius map on $\eC$.
Furthermore, $\C\setminus f(\Ec)$ is polar.
\end{thm}
As a corollary, we get the following result on automorphisms of $\Ec$.
\begin{cor}\label{Aut}
Let $E$ be a closed polar subset of $\C$.
Then $\aut(\Ec)$ is a subgroup of the M\"{o}bius group.
Moreover, any $f\in\aut(\Ec)$ satisfies $f(\eE) = \eE$.
\end{cor}

In light of the above corollary, it is natural to ask how big $\aut(\Ec)$ is, as a subgroup of the M\"{o}bius group. For a generic closed polar set $E$ we would expect $\aut(\Ec)$ to be quite small since $f(\eE) = \eE$ places too many constraints on the M\"{o}bius map $f$.  Here are two standard and well known ways to capture this smallness.

 Note that $\aut(\Ec)$ is a closed subgroup of the locally compact group $SL(2, \mathbb{C})$. Let $\mu$ denote the canonical Haar measure on $SL(2, \mathbb{C})$. A result of Steinhaus states that a closed subgroup of a locally compact group having positive Haar measure is also open. Since $SL(2, \mathbb{C})$ is connected, we have $\mu(\aut(\Ec)) = 0$.  This does not entirely capture the smallness of $\aut(\Ec)$  because the same considerations would also yield that $\mu\left(\aut(\mathbb{D})\right) = 0$ even though the group is large.
    
A finer way to capture the size of $\aut(\Ec)$ is to consider its dimension as a Lie group.
It is known that that $\aut(\Ec)$ is always a discrete group (see \cite{Tsuji}*{Theorem X.48}), hence its dimension is $0$.
In contrast, $\aut(\mathbb{D})$ is a $3$ dimensional real Lie group.

\section{Proofs}\label{sec:proofs}
We now prove Theorem~\ref{thm:CptMainResult} via several lemmas. 

\begin{lemma}\label{lemma:Removable}
Let $E$ be a closed polar subset of $\C$ and $f\in\O(\Ec)$.
If $z_0\in\eE$ is a removable singularity of $f$, then $f$ extends to a holomorphic function in a neighbourhood of $z_0\in\eC$.
\end{lemma}
\begin{proof}
We first use the fact that closed polar sets are removable for bounded harmonic functions, see \cite{Ra}*{Corollary $3.6.2$}. Applying this to $u = \operatorname{Re} f$ and $v = \operatorname{Im} f$, we obtain that $u$ and $v$ extend harmonically to a full neighbourhood $U$ of $z_0$. We already know that $u$ and $v$ satisfy the Cauchy-Riemann equations on $U\setminus{E}.$ Since polar sets are of measure zero and harmonic functions are smooth, it follows that $u$ and $v$ satisfy the Cauchy-Riemann equations on all of $U.$
\end{proof}

\begin{lemma}\label{lemma:OnePole}
Let $E$ be a closed polar subset of $\C$, and let $f$ be a conformal map on $\Ec$.
Suppose $z_0$ is an isolated point of $\eE$ that is a pole of $f$.
Then, no other point of $\eE$ can be a pole-like singularity of $f$.
\end{lemma}
\begin{proof}
Since $z_0$ is a pole, the open mapping theorem guarantees that $f$ maps a punctured neighborhood $U_{z_0}$ of $z_0$ onto a punctured neighborhood $U_{\infty}$ of infinity.
Suppose that $w_0\neq z_0$ is a pole-like singularity of $f$. Then by definition, for every sequence $w_n\in\mathbb{C}\setminus{E}$, with $w_n\to w_0$, we have $f(w_n)\to\infty$. Let $W$ be a neighborhood of $w_0$ that is disjoint from $U_{z_0}$. Choosing a sequence $w_n\in W\cap(\C\setminus{E})$ with $w_n\to w_0$, we observe that for all large $n$, $f(w_n)$ belongs to $U_{\infty}$.
Hence there are points in $U_{\infty}$ which have at least two distinct preimages under $f$, one in $U_{z_0}$ and another in $W$.
The injectivity of $f$ now gives the desired contradiction. This proves that no other point of $\eE$ can be a pole-like singularity of $f$.
\end{proof}

\begin{lemma}\label{lemma:NoEssLikeE}
Let $E$ be a closed polar subset of $\C$ and $f$ be a conformal map on $\Ec$.
Then, no point of $\eE$ is an essential-like singularity of $f$.
\end{lemma}
\begin{proof}
Suppose not. Then, there exist $z_0\in \eE$, sequences $(z_n)$ and $(\widetilde{z}_n)$ in $\Ec$ converging to $z_0$, and $\ell\in\C$ such that $f(z_n)\rightarrow\ell \in \mathbb{C}$ and $f(\widetilde{z}_n) \rightarrow \infty$.

Let $\Omega=f(\Ec)$.
By \cite{Ra}*{Corollary $3.6.8$}, we know that $\C\setminus\Omega$ is polar.
We use the following claim to get a contradiction.
\begin{center}
\begin{minipage}[t]{0.75\linewidth}
{\it Claim:} For every $\tau\in(\abs{\ell}, \infty)$, there exists a sequence $(w_n)$ in $\Ec$ converging to $z_0$ such that
\[
\abs{f(w_n)}=\tau \text{ for all } n,
\quad\text{and}\quad
\lim\limits_{n\to\infty} f(w_n)\in\C\setminus\Omega.
\]
\end{minipage}
\end{center}
Assume the claim is true for now.
Then, the radial projection of $\C\setminus\Omega$ contains the interval $(\abs{\ell},\infty)$.
Since radial projection preserves polar sets, see Theorem 5.3.1 in \cite{Ra}, this contradicts $\C\setminus\Omega$ being polar.

We now prove the claim above.
For all $r>0$, there exists $z_{n_r}, \widetilde{z}_{n_r} \in D_r(z_0)$ such that
\[\abs{f(z_{n_r})}  < \tau < \abs{f(\widetilde{z}_{n_r})}\]
for some sufficiently large $n_r$.
Note $D_r(z_0)\setminus E$ is open and connected, see Theorem 3.6.3 in \cite{Ra} for the latter fact.
Hence, there exists a path $\gamma_r$ in  $D_r(z_0)\setminus E$ that connects $z_{n_r}$ and $\widetilde{z}_{n_r}$.
The path $f(\gamma_r)$ lies in $\Omega$ and there exists $w_r\in \gamma_r$  such that $\abs{f(w_r)}=\tau$.
Letting $r=1/k$ and denoting $w_{1/k}$ by $w_k$, we see that $(w_k)$ is a sequence in $\Ec$ converging to $z_0$, with $\abs{f(w_k)}=\tau$ for all $k$.
By passing to a subsequence if needed, it follows that $(f(w_k))$ is convergent.
The limit of this sequence must lie in $\C\setminus\Omega$, for otherwise, the open mapping theorem will contradict the fact that $f$ is conformal. 
\end{proof}

\begin{lemma}\label{lemma:Discretepole}
With the same notations as above, let $f$ be conformal on $\Ec$.
Then, there exists a discrete countable subset $P\subseteq \eE$ such that every point in $\eE\setminus P$ is removable for $f$.
\end{lemma}
\begin{proof}
By Lemma~\ref{lemma:NoEssLikeE}, all points of $\eE$ are either removable or pole-like singularities of $f$.
So, by Lemma~\ref{lemma:Removable}, we may assume that all points of $\eE$ are pole-like singularities of $f$ and it suffices to show that $\eE$ is a discrete countable set.

Suppose to the contrary there exists a sequence $(z_n)$ of distinct points of $\eE$, with $z_n\to z_0\in \eE$.
Without loss of generality we may assume $0\not\in f(\Ec)$.
Then, there exists a neighborhood $U$ of $z_0$ such that the holomorphic function $g = 1/f$ is bounded above on $U\setminus{\eE}$.
Lemma~\ref{lemma:Removable} guarantees that $g$ extends to be holomorphic in $U$ and vanishes at each $z_n$ and at $z_0$.
The identity theorem guarantees that $g\equiv 0$ in $U$.
This violates the conformality of $f$.
Hence $\eE$ is discrete and countable.
\end{proof}

We are now ready to prove our main result.
\begin{proof}[Proof of Theorem~\ref{thm:CptMainResult}]
Just as earlier, let $\Omega=f(\Ec)$ with $\C\setminus\Omega$ being polar.
We will show that $f$ is a M\"{o}bius map.
To start with, note that by Lemma~\ref{lemma:Discretepole} there exists a discrete countable subset $P$ of $\eE$ such that all points of $\eE\setminus P$ are removable for $f$.
By Lemma~\ref{lemma:NoEssLikeE} there are no essential-like singularities for $f$.
This means $P$ consists of only pole-like singularities.
But since these are isolated, it means points of $P$ are honest poles.
It is easy to see that $P$ is nonempty.
Now, Lemma~\ref{lemma:OnePole} guarantees that $|P| = 1$.

If $P=\{\infty\}$, then $f$ extends to be an injective entire function and therefore of the form $f(z) = az + b$ for $a,b\in\C$ with $a\neq 0$.

On the other hand, if $P=\{z_0\}$ for some $z_0\in\C$, then $f$ has a solitary pole at $z_0\in E$ and every other point of $\eC$ is removable for $f$.
In particular, $f$ is meromorphic on $\eC$.
It is well known that meromorphic functions on the sphere are rational functions.
Injectivity ensures that the degree of the rational function $f$ is $1$.
In other words, $f$ is a M\"{o}bius map.
\end{proof}

After an earlier version of this article was made public, Malik Younsi suggested an alternate approach to Theorem \ref{thm:CptMainResult} which we give below.
We thank him for giving us permission to include his proof here.

In what follows, we change our viewpoint and consider compact polar sets $E$ on the sphere $\eC$.
We recall that a set $E$ is polar on the sphere $\eC$ if $E\setminus{\{\infty\}}$ is polar in $\C$.

\begin{thm}\label{alternate}
Let $E$ be a compact polar set on the sphere $\eC$. Let $f$ be a conformal map on $\eC\setminus{E}$. Then, $f$ extends to a M\"{o}bius map on $\eC$.
\end{thm}
\begin{proof}
We first assume that $E$ is compact in the plane $\C.$ Then, infinity is an isolated singularity of $f$. Using the injectivity of $f$ and a standard Casaroti-Weierstrass argument, it is easy to see that infinity can only be a pole or a removable singularity. After composing with a  M\"{o}bius map, we may assume without loss of generality that $f(\infty) = \infty.$ This implies that $f$ is bounded near $E$, for otherwise injectivity of $f$ will be violated. We now fix $a\in\C\setminus{E}$ and consider the function $g(z) = \frac{f(z)- f(a)}{z-a}$ on $\C\setminus{E}$. Then $g$ is holomorphic in $\eC\setminus{E}$ and bounded both at infinity and near $E$. We now use the fact that a polar set in the plane has zero one dimensional Hausdorff measure, see \cite{Tsuji}. By Painlev\'{e}'s theorem, see \cite{Younsi} for instance, sets with zero one dimensional Hausdorff measure are removable for bounded holomorphic functions. This means that $g$ extends to a bounded entire function on the plane, and thus constant by Liouville's theorem. Unraveling the definition of $g$, we get that $f(z) = cz+ d$ for some complex number $c\neq 0$. 

Suppose now $\infty\in E.$ Fix $z_0\in\eC\setminus{E},$ and consider the map $T(z) = \frac{1}{z-z_0}.$ Then $\widetilde E= T(E)$ is a compact polar set in the plane. Applying the proof of the previous paragraph to the function $\widetilde f = f\circ T^{-1}$ on $\eC\setminus{\widetilde E},$ we get our result. 
\end{proof}


\def\MR#1{\relax\ifhmode\unskip\spacefactor3000 \space\fi%
\href{http://www.ams.org/mathscinet-getitem?mr=#1}{MR#1}}

\end{document}